\documentclass[12pt]{article}
\usepackage{tikz, amsmath, amscd, amssymb, amsfonts, stmaryrd, vmargin, amsthm, soul, bm, graphicx, float, exscale, relsize, etoolbox, mathrsfs, url}
\usepackage[utf8]{inputenc}

\theoremstyle{plain}        \newtheorem{thm}{Theorem}

\theoremstyle{plain}        \newtheorem{pro}[thm]{Proposition}

\theoremstyle{plain}        \newtheorem{lem}[thm]{Lemma}

\theoremstyle{plain}        \newtheorem{cor}[thm]{Corollary}

\theoremstyle{plain}        \newtheorem{rem}[thm]{Remark}

\theoremstyle{plain}        

\theoremstyle{plain}        \newtheorem{defn}{Definition}

\numberwithin{equation}{section}
\numberwithin{thm}{section}

\usepackage{comment}
\usepackage[utf8]{inputenc}
\usepackage[english]{babel}
 
\usepackage{xcolor}

\begin{document}

\title{\bf Bergman projection on the symmetrized bidisk
}

\author{Liwei Chen, Muzhi Jin, 
Yuan Yuan\footnote{ Supported by National Science Foundation grant DMS-1412384, Simons Foundation grant \#429722 and CUSE grant program at Syracuse University.}
}

\date{}

\maketitle

\begin{abstract}
We apply the Bekoll\'e-Bonami estimate for the (positive) Bergman projection on the weighted $L^p$ spaces on the unit disk.
As the consequences, we obtain the boundedness of the Bergman projection on the weighted Sobolev space on the symmetrized bidisk. We also improve the boundedness result of the Bergman projection on the unweighted $L^p$ space on the symmetrized bidisk in \cite{CKY}.
\end{abstract}

\section{Introduction}


Let $\Omega$ be a bounded domain in $\mathbb{C}^{n}$. 
The Bergman projection $\mathcal{B}_{\Omega}$ on $\Omega$ is the orthogonal projection  from  $L^{2}(\Omega)$ to its subspace $A^{2}(\Omega)$--the set of square integrable holomorphic functions, defined by 
\begin{equation*}
\mathcal{B}_{\Omega}(f)(w)=\int_{\Omega} B(w,\eta) f(\eta) dv(\eta)
\end{equation*}
for any $f\in L^{2}(\Omega)$, where $B(w,\eta)$ is the Bergman kernel on $\Omega\times\Omega$.

The Bergman projection $\mathcal{B}_{\Omega}$ in $L^{p}(\Omega)$ space, 
Sobolev spaces $W^{k,p}(\Omega)$ and corresponding weighted spaces is closely related to the $\bar\partial$-Neumann problem on $\Omega$ (cf. \cite{BS2, St2}),
and the regularity problem of $\mathcal{B}_{\Omega}$ is one of the classical problems in several complex variables. It is well known that the $L^2$ Sobolev regularity of the Bergman projection implies the $L^2$ Sobolev regularity of the $\bar\partial$-Neumann operator on a bounded smooth pseudoconvex domain $\Omega$ \cite{BS1}. 
On the other hand, on the worm domain $\Omega$, Barrett showed that $\mathcal{B}_{\Omega}$ does not preserve $W^{k, 2}$ for  large $k$ \cite{Ba}. This was used by Christ to prove the failure of the global regularity of the $\bar\partial$-Neumann operator \cite{Ch}.

 When $\Omega$ is a bounded smooth domain with various convexity conditions on the boundary, the $L^p$ Sobolev regularity of $\mathcal{B}_{\Omega}$ has been intensively studied for general $p \in (1, \infty)$  (cf. \cite{PS, NRSW, MS}). When $\Omega$ is not smooth, the study of the $L^p$ regularity of $\mathcal{B}_{\Omega}$  has also attracted substantial attention in recent years (see for example \cite{LS, KP, Ze, CZ, Che2, EM1, BCEM}). 

The symmetrized bidisk is an interesting model of non-smooth domains and various analytic and geometric properties have been studied intensively (see for example \cite{AY1, AY2, ALY}).
Let $\Phi$ be the rational proper holomorphic map from the bidisk $\mathbb{D}\times\mathbb{D}$ to $\mathbb{C}^2$ given by $\Phi(w_1,w_2)=(w_1+w_2, w_1 w_2)$ with the  determinant of Jacobian $J_{\mathbb{C}} \Phi (w)=w_1-w_2$.
The symmetrized bidisk is then the image of $\mathbb{D}\times\mathbb{D}$ under  $\Phi$ given by
$$\mathbb{G}=\{(w_1+w_2, w_1 w_2)\in\mathbb{C}^{2}\mid(w_1,w_2)\in\mathbb{D}\times\mathbb{D}\}.$$
 Let $(z_1, z_2)$ be the coordinate on $\mathbb{G}$ and let $\delta (z_1,z_2)=-\log |z_{1}^{2}-4z_2|$ be the weight function on $\mathbb{G}$ with $\Phi^{*}\delta=-2\log |w_1-w_2|$.
The norm of weighted Sobolev space $W^{k,p} (\mathbb{G}, l\delta)$ on $\mathbb{G}$ is given by
\begin{equation}\label{1eq1}
\|f\|_{W^{k,p} (\mathbb{G}, l\delta)}^{p}=\sum_{|\alpha|\leq k}\int_{\mathbb{G}} |D_{z,\bar{z}}^{\alpha}(f(z))|^{p}e^{-l\delta}\,d\,v(z)=\sum_{|\alpha|\leq k}\int_{\mathbb{G}} |D_{z,\bar{z}}^{\alpha}(f(z))|^{p}|z_{1}^{2}-4z_2|^{l}\,d\,v(z),
\end{equation}
where $D_{z,\bar{z}}^{\alpha}=\frac{\partial^{\alpha_1+\alpha_2+\alpha_3+\alpha_4}}{\partial z_{1}^{\alpha_1} \partial \bar{z}_{1}^{\alpha_2} \partial z_{2}^{\alpha_3} \partial \bar{z}_{2}^{\alpha_4}}$ with multi-index $\alpha=(\alpha_1,\alpha_2,\alpha_3,\alpha_4)$. When $k=0$, the norm of weighted $L^p$ space is defined similarly.

In \cite{CKY}, the first author, Krantz, and the third author derived the $L^p$ regularity of the Bergman projection $\mathcal{B}_{\mathbb G}$ on the symmetrized bidisk $\mathbb{G}$ by considering the boundness of Bergman projection on $L^p$ functions over upper-half complex plane. In this article, we study the $L^p$ regularity of $\mathcal{B}_{\mathbb G}$ by considering Bekoll\'e-Bonami constant of certain weight on the unit disk $\mathbb{D}$ in complex plane. For the Sobolev regularity, we follow the strategy of holomorphic integration by parts (cf. \cite{Bo, St1, Che1, EM2}).  In the earlier works (cf. \cite{Che1}), the holomorphic integration by parts will reduce the Sobolev estimates of the Bergman projection to the $L^p$ estimate of the Bergman projection as the weight function is annihilated by the tangential vector field. However, in the case of the symmetrized bidisk, the weight function cannot be annihilated by the tangential vector field. Instead we reduce the Sobolev regularity of $\mathcal{B}_{\mathbb G}$ to the boundness of $\mathcal{B}^+_{\mathbb{D}}$ on the weighted $L^p$ space on the unit disk $\mathbb{D}$ (see the definition of  $\mathcal{B}^+_{\mathbb{D}}$  in section 2) and study those weighted $L^p$ regularities by checking the Bekoll\'e-Bonami constant of the corresponding weights.

We would like to point out that the method in this paper would have been used to derive the weighted Sobolev estimates of the Bergman projection on other domains, for instance, the symmetrized polydisk or even a domain covered by the product of planar domains through a holomorphic rational proper map. However, we choose to merely treat the symmetrized bidisk here in order to demonstrate the idea.

The main result is the following theorem on the weighted Sobolev regularity.

\begin{thm}\label{1thm0}
Given $|w_1-w_2|^2$ as a weight function on $w_1 \in \mathbb{D}$,
if the norm of $\mathcal{B}^+_{\mathbb{D}}: L^p \left(\mathbb{D}, |w_1-w_2|^2\right) \rightarrow L^p\left(\mathbb{D}, |w_1-w_2|^2\right)$ is uniformly bounded independent of $w_2 \in \mathbb{D}$ for some $p \in (1, \infty)$, 
then  the Bergman projection $\mathcal{B}_{\mathbb{G}}$ on the symmetrized bidisk $\mathbb{G}$ is bounded
 from $W^{k,p} (\mathbb{G})$ to $W^{k,p} \left(\mathbb{G}, \frac{3kp\delta}{2}\right)$ for any positive integer $k$.
   \end{thm}

In order to obtain an explicit range in $p$ such that $\mathcal{B}_{\mathbb{G}}$ is bounded between weighted Sobolev spaces and $L^p$ spaces, we obtain the following results.

\begin{pro}\label{5cor1}
Let $2<p<+\infty$.
For any fixed $w_2\in \mathbb{D}$, $\mathcal{B}^+_{\mathbb{D}}: L^p \left(\mathbb{D}, |w_1-w_2|^2\right) \rightarrow L^p\left(\mathbb{D}, |w_1-w_2|^2\right)$ is bounded. Moreover, the norm is uniformly bounded independent of $w_2$.
\end{pro}


\begin{pro}\label{1cor1}
Let $\frac{4}{3}<p<4$. For any fixed $w_2\in \mathbb{D}$, $\mathcal{B}_{\mathbb{D}}: L^p \left(\mathbb{D}, |w_1-w_2|^{2-p}\right) \rightarrow L^p\left(\mathbb{D}, |w_1-w_2|^{2-p}\right)$ is bounded. Moreover, the norm is uniformly bounded independent of $w_2$.
\end{pro}

As the straightforward corollary, we have:

\begin{cor}  \label{Lp}
\begin{itemize}
\item For $2<p<+\infty$, the Bergman projection $\mathcal{B}_{\mathbb{G}}$ is bounded  from $W^{k,p} (\mathbb{G})$ to $W^{k,p} \left(\mathbb{G}, \frac{3kp\delta}{2}\right)$ for any positive integer $k$.
\item For $\frac{4}{3}<p<4$,  the Bergman projection $\mathcal{B}_{\mathbb{G}}$ is bounded  from $L^p(\mathbb{G})$ to $L^p(\mathbb{G})$.
\item For $\frac{4}{3} < p < \infty$, the Bergman projection $\mathcal{B}_{\mathbb{G}}$ is bounded from $L^{p} (\mathbb{G})$ to $L^{p} \left(\mathbb{G}, \frac{p\delta}{2}\right)$.
\end{itemize}
\end{cor}



\begin{rem}
By the similar method, one can prove that the Bergman projection on the $n$-dimensional symmetrized polydisk $\mathbb{G}^n$  is $L^p$ bounded for $p \in \left( 1+ \frac{n-1}{n+1}, 1+ \frac{n+1}{n-1} \right), $ which is a slight improvement of Theorem 4.9 in \cite{CKY}. Since the case $n=2$ is proved in details, we only demonstrate the difference in the general dimension as follows. Firstly, notice that Proposition \ref{1thm00} for $\mathbb{G}^n$ still holds with the weight $|w_1 - w_2|^{2-p}$ replaced by $\left| \prod_{i < j} (w_i - w_j) \right|^{2-p}.$ Secondly, combining Lemma 2.4 in \cite{CKY} and Theorem \ref{thmbb}, it suffices to show $\mathcal{B}_{\mathbb{D}}: L^p \left(\mathbb{D}, |w- \xi|^{\frac{2-p}{n-1}}\right) \rightarrow L^p\left(\mathbb{D}, |w- \xi|^{\frac{2-p}{n-1}}\right)$ is bounded for any $\xi \in \mathbb{D}$ and the norm is uniformly bounded independent of $\xi$. Lastly, for $p \in \left( 1+ \frac{n-1}{n+1}, 1+ \frac{n+1}{n-1} \right)$, the boundedness of $\mathcal{B}_{\mathbb{D}}$ in $L^p \left(\mathbb{D}, |w- \xi|^{\frac{2-p}{n-1}}\right)$ can be proved using the similar argument as in the proof of Proposition \ref{1cor1}.
\end{rem}

\section{The Bergman projection on the weighted $L^p$ space}

Let $B(w,\eta)$ be the Bergman kernel on $\Omega\times\Omega$, then define the operator $\mathcal{B}_{\Omega}^{+}$ by
\begin{equation*}
\mathcal{B}_{\Omega}^{+}(f)(w)=\int_{\Omega} |B(w,\eta)| f(\eta) dv(\eta),
\end{equation*}
for any $f\in L^{2}(\Omega)$. For the purpose of the present article, we only consider  Bekoll\'e-Bonami's result of the Bergman projection on the weighted $L^p$ space over the unit disk $\mathbb{D}$. 
%
%
Let $T_z$ denote the Carleson tent defined as:
$$T_z :=\left\{w\in\mathbb{D}: \left|1-\overline{w}\frac{z}{|z|}\right|<1-|z|\right\}\qquad\text{for}\,\,\, z\in\mathbb{D}\setminus\{0\},$$
and $T_z :=\mathbb{D}$ when $z=0$.
Note that for $z\neq 0$, 
$T_z$
is the intersection of the unit disk and a disk centered at a point $\frac{z}{|z|}$ on the unit circle with radius $R=1-|z|<1$. By elementary geometry, it can be shown that $\int_{T_z}\,dA(w)\approx (1-|z|)^2=R^2$ (cf. Lemma 2.1 in \cite{HW}).

In \cite{BB}, Bekoll\'e and Bonami proved an well-known regularity result of the Bergman projection on weighted $L^p$ space over the unit disk. Here we will apply the following formulation due to \cite{RTW}. 
Note that there are extensive recent studies on the Bekoll\'e-Bonami estimates (cf. \cite{HW, HWW1, HWW2} and the references therein).

\begin{thm}[\cite{RTW}]\label{thmbb}
Let the weight $\sigma$ be a positive, local integrable function on $\mathbb{D}$ and let $1<p<\infty$. Then 
\begin{align*}
\left(B_p(\sigma)\right)^{\frac{1}{2p}}
&\lesssim \|B_{\mathbb{D}}: L^{p}(\mathbb{D},\sigma)\rightarrow L^{p}(\mathbb{D},\sigma)\| \\
&\leq \|B_{\mathbb{D}}^{+}: L^{p}(\mathbb{D},\sigma)\rightarrow L^{p}(\mathbb{D},\sigma)\|  \lesssim\left(B_p(\sigma)\right)^{\max\{1, \frac{1}{p-1}\}},
\end{align*}
where $$B_p(\sigma):=\sup_{z\in\mathbb{B}_n}\frac{\int_{T_z}\sigma(w)\,dA(w)}{\int_{T_z}\,dA(w)}\left(\frac{\int_{T_z}\sigma^{-\frac{1}{p-1}}(w)\,dA(w)}{\int_{T_z}\,dA(w)}\right)^{p-1}.$$

\end{thm}



Now we are going to prove Proposition \ref{5cor1}  and Proposition \ref{1cor1} by checking the corresponding Bekoll\'e-Bonami constants. Both of the proofs are similar to the corresponding arguments in \S4 in \cite{CKY}, but we still include them here for the completeness.

\begin{proof}[Proof of Proposition \ref{5cor1}]
We only show the case $z\neq 0$ as the case $z=0$ is similar. Let $L=\text{dist}\left(w_2, \frac{z}{|z|}\right)$. We will prove the uniform boundedness of $B_p\left(|w_1-w_2|^{2}\right)$ for different types of disks.
 
 Assume $L\geq 10R$, then $9R\leq |w_1-w_2|\leq 11R$ for any $w_1\in T_z$. It follows that
 $$\frac{\int_{T_z} |w_1-w_2|^{2} dA(w_1)}{\int_{T_z}\,dA(w_1)}\cdot \left( \frac{\int_{T_z} |w_1-w_2|^{-\frac{2}{p-1}} dA(w_1)}{\int_{T_z}\,dA(w_1)}\right)^{p-1}\leq \left(\frac{11}{9}\right)^{2}.$$
 
 Assume $L< 10R$. We split our argument into two different cases. For $0<R<\delta$, where $\delta>0$ is sufficiently small, then $T_z\subset D':=D(w_2; 20R)$, the disc centered at $w_2$ with radius $20R$. It follows that
 
$$\int_{T_z} |w_1-w_2|^{2} dA(w_1)\leq \int_{D'} |w_1-w_2|^{2} dA(w_1)\approx R^{4};$$ 
and 
$$\int_{T_z} |w_1-w_2|^{-\frac{2}{p-1}} dA(w_1)\leq \int_{D'} |w_1-w_2|^{-\frac{2}{p-1}} dA(w_1)\approx R^{\frac{2p-4}{p-1}}$$ if $p>2$. Therefore

$$\frac{\int_{T_z} |w_1-w_2|^{2} dA(w_1)}{\int_{T_z}\,dA(w_1)}\cdot \left( \frac{\int_{T_z} |w_1-w_2|^{-\frac{2}{p-1}} dA(w_1)}{\int_{T_z}\,dA(w_1)}\right)^{p-1}\approx \frac{R^{4}}{R^2}\cdot \left(\frac{R^{\frac{2p-4}{p-1}}}{R^2}\right)^{p-1}=1,$$
independent of $w_2\in\mathbb{D}$ provided $p>2$. On the other hand, for $\delta<R<1$, then $T_z\subset \mathbb{D}\subset D:=D(w_2;2)$ and $\int_{T_z}\,dV(w_1)\geq C_{\delta} >0$. 
So
$$\int_{T_z} |w_1-w_2|^{2} dA(w_1)\leq \int_{D} |w_1-w_2|^{2} dA(w_1)\lesssim 1;$$ and

$$\int_{T_z} |w_1-w_2|^{-\frac{2}{p-1}} dA(w_1)\leq \int_{D} |w_1-w_2|^{-\frac{2}{p-1}} dA(w_1)\lesssim 1$$ if $p>2$.
Therefore, 
$$\frac{\int_{T_z} |w_1-w_2|^{2} dA(w_1)}{\int_{T_z}\,dA(w_1)}\cdot \left( \frac{\int_{T_z} |w_1-w_2|^{-\frac{2}{p-1}} dA(w_1)}{\int_{T_z}\,dA(w_1)}\right)^{p-1}\leq \frac{\text{Constant}}{C_{\delta}^p}$$ independent of $w_2\in\mathbb{D}$ provided $p>2$.

 Hence, for $2<p<\infty$, the proposition  is proved as 
 the Bekoll\'e-Bonami constant is uniformly bounded independent of $w_2\in\mathbb{D}$ by Theorem \ref{thmbb}.\end{proof}

\begin{proof}[Proof of Proposition \ref{1cor1}]
We only show the case $z\neq 0$ as the case $z=0$ is similar. Let $L=\text{dist}\left(w_2, \frac{z}{|z|}\right)$. We will prove the uniform boundedness of $B_p\left(|w_1-w_2|^{2-p}\right)$ for different types of disks. 
 
 Assume $L\geq 10R$, then $9R\leq |w_1-w_2|\leq 11R$ for any $w_1\in T_z$. When $\frac{4}{3}<p\leq 2$, $|w_1-w_2|^{2-p}\leq (11R)^{2-p}$ and $|w_1-w_2|^{-\frac{2-p}{p-1}}\leq (9R)^{-\frac{2-p}{p-1}}$. So
 $$\frac{\int_{T_z} |w_1-w_2|^{2-p} dA(w_1)}{\int_{T_z}\,dA(w_1)}\cdot \left( \frac{\int_{T_z} |w_1-w_2|^{-\frac{2-p}{p-1}} dA(w_1)}{\int_{T_z}\,dA(w_1)}\right)^{p-1}\leq \left(\frac{11}{9}\right)^{2-p}.$$
 When $2\leq p<4$, $|w_1-w_2|^{2-p}\leq (9R)^{2-p}$ and $|w_1-w_2|^{-\frac{2-p}{p-1}}\leq (11R)^{-\frac{2-p}{p-1}}$. So
 $$\frac{\int_{T_z} |w_1-w_2|^{2-p} dA(w_1)}{\int_{T_z}\,dA(w_1)}\cdot \left( \frac{\int_{T_z} |w_1-w_2|^{-\frac{2-p}{p-1}} dA(w_1)}{\int_{T_z}\,dA(w_1)}\right)^{p-1}\leq \left(\frac{9}{11}\right)^{2-p}.$$
 
 Assume $L< 10R$. Similarly, we split our argument into two cases. For $0<R<\delta$, where $\delta>0$ is a sufficiently small constant, then $T_z\subset D':=D(w_2; 20R)$. It follows that
 
$$\int_{T_z} |w_1-w_2|^{2-p} dA(w_1)\leq \int_{D'} |w_1-w_2|^{2-p} dA(w_1)\approx R^{4-p}$$ 
 if $p<4$; and
 $$\int_{T_z} |w_1-w_2|^{-\frac{2-p}{p-1}} dA(w_1)\leq \int_{D'} |w_1-w_2|^{-\frac{2-p}{p-1}} dA(w_1)\approx R^{\frac{3p-4}{p-1}}$$  if $p>\frac{4}{3}$; 
and thus
$$\frac{\int_{T_z} |w_1-w_2|^{2-p} dA(w_1)}{\int_{T_z}\,dA(w_1)}\cdot \left( \frac{\int_{T_z} |w_1-w_2|^{-\frac{2-p}{p-1}} dA(w_1)}{\int_{T_z}\,dA(w_1)}\right)^{p-1}\approx \frac{R^{4-p}}{R^2}\cdot \left(\frac{R^{\frac{3p-4}{p-1}}}{R^2}\right)^{p-1}=1,$$
 independent of $w_2\in\mathbb{D}$ for $\frac{4}{3}<p<4$.
 On the other hand, for $\delta<R<1$, then $T_z\subset \mathbb{D}\subset D:=D(w_2;2)$ and $\int_{T_z}\,dV(w_1)\geq C_{\delta} >0$. 
It follows that 
 
 $$\int_{T_z} |w_1-w_2|^{2-p} dA(w_1)\leq \int_{D} |w_1-w_2|^{2-p} dA(w_1)\lesssim 1$$ if $p<4$; and 

$$\int_{T_z} |w_1-w_2|^{-\frac{2-p}{p-1}} dA(w_1)\leq \int_{D} |w_1-w_2|^{-\frac{2-p}{p-1}} dA(w_1)\lesssim 1$$ if  $p>\frac{4}{3}$; and 
hence, 
$$\frac{\int_{T_z} |w_1-w_2|^{2-p} dA(w_1)}{\int_{T_z}\,dA(w_1)}\cdot \left( \frac{\int_{T_z} |w_1-w_2|^{-\frac{2-p}{p-1}} dA(w_1)}{\int_{T_z}\,dA(w_1)}\right)^{p-1}\leq \frac{\text{Constant}}{C_{\delta}^p},$$
  independent of $w_2\in\mathbb{D}$ for $\frac{4}{3}<p<4$. Therefore, the proposition is proved.

 
\end{proof}



\section{Weighted Sobolev regularity of $\mathcal{B}_{\mathbb{G}}$}
\subsection{Transferring the data to the product space}
Functions on $\mathbb{G}$ can be transferred to functions on  $\mathbb{D}\times\mathbb{D}$ by the pull-back by $\Phi$. Also the Bergman kernel $B_\mathbb{G}$ on $\mathbb{G}\times\mathbb{G}$ has the following representation under the change of coordinates (cf. \cite{CKY}, or \cite{EZ}, \cite{MSZ}).
\begin{pro}\label{2pro1}
The Bergman kernel $B_\mathbb{G}$ on $\mathbb{G}\times\mathbb{G}$ has the following representation with coordinate $(w_1,w_2,\eta_1,\eta_2)\in \mathbb{D}\times \mathbb{D}\times \mathbb{D}\times \mathbb{D}$.
\begin{equation}
\begin{split}
B_\mathbb{G}(\Phi(w),\Phi(\eta))&=\frac{1}{2\pi^2}\cdot\frac{1}{w_1-w_2}\cdot\frac{1}{\bar{\eta}_1-\bar{\eta}_2}[\frac{1}{(1-w_{1}\bar{\eta}_1)^{2}}\cdot\frac{1}{(1-w_{2}\bar{\eta}_2)^{2}}\\
&~~~~~~~~~~~~~~  -\frac{1}{(1-w_{1}\bar{\eta}_2)^{2}}\cdot\frac{1}{(1-w_{2}\bar{\eta}_1)^{2}}].
\end{split}
\end{equation}
\end{pro}

On the other hand, since we are going to consider the Sobolev norm of the Bergman projection of functions, any differential operator with the anti-holomorphic direction acting on the holomorphic functions vanishes. Thus we only need to consider the holomorphic differential operators $D_{z}^{\alpha}:=\frac{\partial^{\alpha_1+\alpha_2}}{\partial z_1^{\alpha_1}  \partial z_2^{\alpha_2}}$ on $\mathbb{G}$ with multi-index $\alpha=(\alpha_1, \alpha_2)$. 

\begin{lem}\label{2lem1}
For any multi-index $\alpha=(\alpha_1, \alpha_2)$ with $\alpha_1\geq 0, \alpha_2 \geq 0, |\alpha|:=\alpha_1+\alpha_2 \geq 1$, 
\begin{equation}\label{112}
D_{z}^{\alpha}=\frac{1}{(w_1-w_2)^{2|\alpha|-1}}\sum_{1\leq |\beta| \leq |\alpha|}P_{\alpha, \beta}(w_1,w_2)\frac{\partial^{\beta_1+\beta_2}}{\partial w_{1}^{\beta_1}\partial w_{2}^{\beta_2}},
\end{equation} where $P_{\alpha, \beta}(w_1,w_2)$ are holomorphic monomials in $w_1, w_2$ with degree at most $2|\alpha| -1$. One the other hand, for any multi-index $\beta=(\beta_1, \beta_2,\beta_3, \beta_4)$,
\begin{equation}\label{123}
D_{w,\bar{w}}^{\beta}=\sum_{0\leq |\alpha| \leq |\beta|} \tilde P_{\alpha, \beta}(w_1, \bar{w}_1, w_2, \bar{w}_2)\frac{\partial^{\alpha}}{\partial z_{1}^{\alpha_1}\partial \bar{z}_{1}^{\alpha_2}\partial z_{2}^{\alpha_3}\partial \bar{z}_{2}^{\alpha_4}},
\end{equation}
where $\tilde P_{\alpha, \beta}(w_1, \bar{w}_1, w_2, \bar{w}_2)$ are polynomials in $w_1, \bar{w}_1, w_2, \bar{w}_2$ with degree at most $|\beta|$.
\end{lem}
\begin{proof}
By the change of coordinates under the holomorphic mapping $\Phi$, we have
\begin{equation}\label{2eq0}
\begin{cases}
\frac{\partial}{\partial w_1}=\frac{\partial}{\partial z_1}+w_2\frac{\partial}{\partial z_2}\\
\frac{\partial}{\partial w_2}=\frac{\partial}{\partial z_1}+w_1\frac{\partial}{\partial z_2},
\end{cases}
\end{equation}
Then (\ref{2eq0}) implies
\begin{equation}
\begin{cases}
\frac{\partial}{\partial z_1}=\frac{w_1}{w_1-w_2}\frac{\partial}{\partial w_1}-\frac{w_2}{w_1-w_2}\frac{\partial}{\partial w_2}\\
\frac{\partial}{\partial z_2}=-\frac{1}{w_1-w_2}\frac{\partial}{\partial w_1}+\frac{1}{w_1-w_2}\frac{\partial}{\partial w_2}.
\end{cases}
\end{equation}
This shows the case when $|\alpha|=1$. Suppose (\ref{112}) holds when $|\alpha|=j$, we prove the case of  $|\alpha|=j+1$ by induction. 
\begin{align*}
\frac{\partial}{\partial z_1}\frac{\partial^{\alpha_1+\alpha_2}}{\partial z_{1}^{\alpha_1}\partial z_{2}^{\alpha_2}}
&=\left(\frac{w_1}{w_1-w_2}\frac{\partial}{\partial w_1}-\frac{w_2}{w_1-w_2}\frac{\partial}{\partial w_2}\right)\\
&~~~~ \left(\frac{1}{(w_1-w_2)^{2(\alpha_1+\alpha_2)-1}}\sum_{0<|\beta|\leq|\alpha|} P_{\alpha, \beta}(w_1,w_2)\frac{\partial^{\beta_1+\beta_2}}{\partial w_{1}^{\beta_1}\partial w_{2}^{\beta_2}}\right)\\
&=\frac{1}{(w_1-w_2)^{2(\alpha_1+\alpha_2)+1}}\sum_{0<|\beta|\leq|\alpha|}(-2(\alpha_1+\alpha_2)+1)w_1  P_{\alpha, \beta}(w_1,w_2)\frac{\partial^{\beta_1+\beta_2}}{\partial w_{1}^{\beta_1}\partial w_{2}^{\beta_2}}\\
& ~~~+\frac{1}{(w_1-w_2)^{2(\alpha_1+\alpha_2)}}\sum_{0<|\beta|\leq|\alpha|}w_1\frac{\partial}{\partial w_1} P_{\alpha, \beta}(w_1,w_2)\frac{\partial^{\beta_1+\beta_2}}{\partial w_{1}^{\beta_1}\partial w_{2}^{\beta_2}}\\
&~~~+\frac{1}{(w_1-w_2)^{2(\alpha_1+\alpha_2)}}\sum_{0<|\beta|\leq|\alpha|}w_1  P_{\alpha, \beta}(w_1,w_2)\frac{\partial^{\beta_1+\beta_2+1}}{\partial w_{1}^{\beta_1+1}\partial w_{2}^{\beta_2}}\\
&~~~+\frac{1}{(w_1-w_2)^{2(\alpha_1+\alpha_2)+1}}\sum_{0<|\beta|\leq|\alpha|}(-2(\alpha_1+\alpha_2)+1)w_2  P_{\alpha, \beta}(w_1,w_2)\frac{\partial^{\beta_1+\beta_2}}{\partial w_{1}^{\beta_1}\partial w_{2}^{\beta_2}}\\
&~~~+\frac{1}{(w_1-w_2)^{2(\alpha_1+\alpha_2)}}\sum_{0<|\beta|\leq|\alpha|}(-w_2)\frac{\partial}{\partial w_2} P_{\alpha, \beta}(w_1,w_2)\frac{\partial^{\beta_1+\beta_2}}{\partial w_{1}^{\beta_1}\partial w_{2}^{\beta_2}}\\
&~~~+\frac{1}{(w_1-w_2)^{2(\alpha_1+\alpha_2)}}\sum_{0<|\beta|\leq|\alpha|}(-w_2)  P_{\alpha, \beta}(w_1,w_2)\frac{\partial^{\beta_1+\beta_2+1}}{\partial w_{1}^{\beta_1}\partial w_{2}^{\beta_2+1}}.\\
\end{align*}
This finishes the induction after factoring out $\frac{1}{(w_1-w_2)^{2(\alpha_1+\alpha_2)+1}}$. 
(\ref{123}) can be proved in a similar manner by induction.
\end{proof}

\subsection{Holomorphic integration by parts}
The holomorphic derivatives of the Bergman kernel can be transformed to anti-holomorphic derivative as follows.
\begin{lem}\label{3lem1}
When $w_i\neq 0$, $\frac{\partial^{\beta}}{\partial w_{i}^{\beta}}\left( \frac{1}{(1-w_{i}\bar{\eta}_j)^{2}}\right)=\frac{\bar{\eta}_{j}^{\beta}}{w_{i}^{\beta}}\cdot\frac{\partial^{\beta}}{\partial \bar{\eta}_{j}^{\beta}}\left( \frac{1}{(1-w_{i}\bar{\eta}_j)^{2}}\right)$ for any $i, j=1,2$.
\end{lem}
\begin{proof}
Let $r=w_{i}\bar{\eta}_j$. By applying the chain rule $\beta$ times, we have 
$$\frac{\partial^{\beta}}{\partial w_{i}^{\beta}}\left( \frac{1}{(1-w_{i}\bar{\eta}_j)^{2}}\right)=\frac{\partial^{\beta}}{\partial r^{\beta}}\left( \frac{1}{(1-w_{i}\bar{\eta}_j)^{2}}\right)\cdot \bar{\eta}_{j}^{\beta}.$$
Similarly, 
$$\frac{\partial^{\beta}}{\partial \bar{\eta}_{j}^{\beta}}\left( \frac{1}{(1-w_{i}\bar{\eta}_j)^{2}}\right)=\frac{\partial^{\beta}}{\partial r^{\beta}}\left( \frac{1}{(1-w_{i}\bar{\eta}_j)^{2}}\right)\cdot w_{i}^{\beta}.$$
The lemma is thus proved.
\end{proof}

The next lemma implies that $\frac{\partial}{\partial \bar\eta}$ can be replaced by the tangential operator and it follows from the straightforward calculation.

\begin{lem}\label{3lem2}
Let $T_{\eta}=\eta\frac{\partial}{\partial \eta}-\bar{\eta}\frac{\partial}{\partial \bar{\eta}}$ be the tangential operator on the disc $\overline{\mathbb{D}}$ and $f$ be anti-holomorphic. 
Then for $\beta\in\mathbb{Z}^{+}\cup \{0\}$, we have
\begin{equation}\label{3eq1}
\bar{\eta}^{\beta}\frac{\partial}{\partial \bar{\eta}^{\beta}} (f)=T_{\eta}^{\beta} (f). 
\end{equation}
\end{lem}


For $\beta \geq 1$, following the idea of the partial Bergman kernel in \cite{EM2}, one can define $K_{\beta}(w,\eta) =\frac{1}{(1-w\bar{\eta})^{2}}-\sum_{j=0}^{\beta-1}(j+1)(w\bar{\eta})^{j}$  as Bergman kernel subtracting the first $\beta$ terms in its Taylor series in $w\bar{\eta}$. Then one obtains 
\begin{equation}\label{111}
\frac{\partial^\beta}{\partial \bar{\eta}^{\beta}} \left(\frac{1}{(1-w\bar{\eta})^{2}}\right)=\frac{\partial^\beta}{\partial \bar{\eta}^{\beta}}K_{\beta}(w,\eta).
\end{equation} 
Moreover, 
\begin{equation}\label{2eq7}\notag
\begin{aligned}
K_{\beta}(w,\eta)
&=\frac{\partial}{\partial(w\bar{\eta})}\left(\sum_{j=\beta}^{\infty}(w\bar{\eta})^{j+1}\right)\\
&=\frac{\partial}{\partial(w\bar{\eta})}\left(\frac{(w\bar{\eta})^{\beta+1}}{1-w\bar{\eta}}    \right)          \\
&=\frac{(\beta+1)(w\bar{\eta})^{\beta}-\beta(w\bar{\eta})^{\beta+1}}{(1-w\bar{\eta})^{2}}.
\end{aligned}
\end{equation} 
It follows that
\begin{equation}\label{4eq6}
|K_{\beta}(w,\eta)|=\left|\frac{(\beta+1)(w\bar{\eta})^{\beta}-\beta(w\bar{\eta})^{\beta+1}}{(1-w\bar{\eta})^{2}}\right|\lesssim \left|\frac{w^{\beta}}{(1-w\bar{\eta})^{2}}\right|.
\end{equation}

\begin{cor}\label{3cor1}
Assume $f(\eta_1,\eta_2)\in W^{k,2}(\mathbb{D}\times\mathbb{D})$ and $w_1, w_2 \in \mathbb{D}^*=\mathbb{D}\setminus\{0\}$, then
\begin{equation*}
\begin{aligned}
&~~~~\int_{\mathbb{D}\times\mathbb{D}} 
 \left[\frac{\partial^{\beta_1}}{\partial w_{1}^{\beta_1}}\left(\frac{1}{(1-w_{1}\bar{\eta}_1)^{2}}\right)\cdot \frac{\partial^{\beta_2}}{\partial w_{2}^{\beta_2}}\left(\frac{1}{(1-w_{2}\bar{\eta}_2)^{2}}\right)\right] f(\eta_1,\eta_2)\,dv(\eta)\\
&=\frac{1}{w_1^{\beta_1} w_2^{\beta_2}}\int_{\mathbb{D}\times\mathbb{D}}K_{\beta_1}(w_1,\eta_1) K_{\beta_2}(w_2,\eta_2)T_{\eta_1}^{\beta_1} T_{\eta_2}^{\beta_2}f(\eta_1,\eta_2)\,dv(\eta).
\end{aligned}
\end{equation*}
\end{cor}
\begin{proof}
By Lemma \ref{3lem1}, Lemma \ref{3lem2}, (\ref{111}) and Fubini theorem, we have
\begin{equation*}
\begin{aligned}
~~~~&~~~~\int_{\mathbb{D}\times\mathbb{D}}  \left[\frac{\partial^{\beta_1}}{\partial w_{1}^{\beta_1}}\left(\frac{1}{(1-w_{1}\bar{\eta}_1)^{2}}\right)\cdot \frac{\partial^{\beta_2}}{\partial w_{2}^{\beta_2}}\left(\frac{1}{(1-w_{2}\bar{\eta}_2)^{2}}\right)\right] f(\eta_1,\eta_2)\,dv(\eta)\\
& =\int_{\mathbb{D}}  \frac{1}{w_1^{\beta_1}}T_{\eta_1}^{\beta_1}\left(K_{\beta_1}(w_1,\eta_1)\right)\left(\frac{1}{w_{2}^{\beta_2}}\int_{\mathbb{D}} T_{\eta_2}^{\beta_2}\left( K_{\beta_2}(w_2,\eta_2)\right)  f(\eta_1,\eta_2)\,dv(\eta_2)\right)\,dv(\eta_1)\\
&= \frac{1}{w_1^{\beta_1}w_{2}^{\beta_2}}  \int_{\mathbb{D}} T_{\eta_1}^{\beta_1}\left(K_{\beta_1}(w_1,\eta_1)\right) \left( \int_{\mathbb{D}}   K_{\beta_2}(w_2,\eta_2) T_{\eta_2}^{\beta_2}\left( f(\eta_1,\eta_2)\right)\,dv(\eta_2)\right)\,dv(\eta_1)\\
&=\frac{1}{w_1^{\beta_1} w_2^{\beta_2}}\int_{\mathbb{D}} K_{\beta_1}(w_1,\eta_1) T_{\eta_1}^{\beta_1}  \left(\int_{\mathbb{D}} K_{\beta_2}(w_2,\eta_2) T_{\eta_2}^{\beta_2}\left( f(\eta_1,\eta_2)\right)\,dv(\eta_2)\right)dv(\eta_1)\\
& =\frac{1}{w_1^{\beta_1} w_2^{\beta_2}}\int_{\mathbb{D}\times\mathbb{D}}K_{\beta_1}(w_1,\eta_1) K_{\beta_2}(w_2,\eta_2)T_{\eta_1}^{\beta_1} T_{\eta_2}^{\beta_2}f(\eta_1,\eta_2)\,dv(\eta),
\end{aligned}
\end{equation*}
where the second and the third equalities follow from the integration by parts.
\end{proof}

\section{Proof of the main results}

\begin{proof}[Proof of Theorem \ref{1thm0}]
For $f\in W^{k,p} (\mathbb{G})$, by Lemma \ref{2lem1}, one sees
\begin{equation}\label{4eq4}
\begin{aligned}
\|\mathcal{B}_\mathbb{G}(f)\|_{W^{k,p} \left(\mathbb{G}, \left(\frac{3kp}{2}\right)\delta\right)}^{p}
& =\sum_{0\leq|\alpha|\leq k}\int_{\mathbb{G}} |D_{z}^{\alpha}\left(\mathcal{B}_\mathbb{G}(f)\right)|^{p}|z_{1}^{2}-4z_2|^{\frac{3kp}{2}}\,d\,v(z)\\
& \leq \int_{\mathbb{G}} |\mathcal{B}_\mathbb{G}(f)|^{p}|z_{1}^{2}-4z_2|^{\frac{3kp}{2}}\,d\,v(z)\\
&~~~ +\sum_{1\leq|\alpha|\leq k} \int_{\mathbb{D}\times\mathbb{D}}\frac{1}{|(w_1-w_2)^{2|\alpha|-1}|^{p}}\sum_{1\leq|\beta|\leq|\alpha|}|P_{\beta,\alpha}(w_1,w_2)|^{p}\\
&~~~~~\cdot \left|\frac{\partial^{\beta_{1}+\beta_{2}}}{\partial w_{1}^{\beta_{1}}\partial w_{2}^{\beta_{2}}}(\mathcal{B}_\mathbb{G}(f)\circ\Phi)\right|^{p}\cdot |w_1-w_2|^{3kp+2}\,dv(w).\\
\end{aligned}
\end{equation}

We first look at the second term on the right hand side of (\ref{4eq4}).
\begin{equation}\label{4eq5}
\begin{aligned}
&~~\sum_{1\leq|\alpha|\leq k} \int_{\mathbb{D}\times\mathbb{D}}\frac{1}{|(w_1-w_2)^{2|\alpha|-1}|^{p}}\sum_{1\leq|\beta|\leq|\alpha|}|P_{\beta,\alpha}(w_1,w_2)|^{p}\\
&~~~~~\cdot \left|\frac{\partial^{\beta_{1}+\beta_{2}}}{\partial w_{1}^{\beta_{1}}\partial w_{2}^{\beta_{2}}}(\mathcal{B}_\mathbb{G}(f)\circ\Phi)\right|^{p}\cdot |w_1-w_2|^{3kp+2}\,dv(w)\\
& \lesssim \sum_{|\beta|\leq k} \int_{\mathbb{D}\times\mathbb{D}} | D_{w}^{\beta}(\mathcal{B}_\mathbb{G}(f)\circ\Phi)|^{p} |w_1-w_2|^{kp+p+2}\,dv(w)\\
& =\frac{1}{2^{p}\pi^{2p}}\sum_{|\beta|\leq k} \int_{\mathbb{D}\times\mathbb{D}}  \left|\int_{\mathbb{D}\times\mathbb{D}}\frac{\partial^{\beta_1+\beta_2}}{\partial w_{1}^{\beta_1}\partial w_{2}^{\beta_2}}\left(\frac{1}{w_1-w_2}\cdot\frac{1}{\bar{\eta}_1-\bar{\eta}_2} \left[\frac{1}{(1-w_{1}\bar{\eta}_1)^{2}}\cdot\frac{1}{(1-w_{2}\bar{\eta}_2)^{2}}\right.\right.\right.\\
&~~~\left.\left.\left.-\frac{1}{(1-w_{1}\bar{\eta}_2)^{2}}\cdot\frac{1}{(1-w_{2}\bar{\eta}_1)^{2}}\right]\right) f(\Phi(\eta))|\eta_1-\eta_2|^{2}\,dv(\eta)\right|^{p}|w_1-w_2|^{kp+p+2}\,dv(w)\\
& \lesssim \sum_{|\beta|\leq k} \int_{\mathbb{D}\times\mathbb{D}} \left|\int_{\mathbb{D}\times\mathbb{D}} \left[\frac{\partial^{\beta_1}}{\partial w_{1}^{\beta_1}}\left(\frac{1}{(1-w_{1}\bar{\eta}_1)^{2}}\right)\cdot \frac{\partial^{\beta_2'}}{\partial w_{2}^{\beta_2}}\left(\frac{1}{(1-w_{2}\bar{\eta}_2)^{2}}\right)\right.\right.\\
&~~~\left.\left.-\frac{\partial^{\beta_1}}{\partial w_{1}^{\beta_1}}\left(\frac{1}{(1-w_{1}\bar{\eta}_2)^{2}}\right)\cdot \frac{\partial^{\beta_2}}{\partial w_{2}^{\beta_2}}\left(\frac{1}{(1-w_{2}\bar{\eta}_1)^{2}}\right)\right]f(\Phi(\eta))(\eta_1-\eta_2)\,dv(\eta)\right|^{p}\\
&~~~~\cdot |w_1-w_2|^{2}\,dv(w).\\
\end{aligned}
\end{equation}
When the derivative applies to $\frac{1}{w_1-w_2}$, the degree will be no less than $-k-1$, which will be absorbed by the weight $|w_1-w_2|^{kp+p+2}$. Here we only consider the one term in (\ref{4eq5}), and the other term can be handled by the same argument. 

\begin{equation}\label{111122}
\begin{aligned}
&~~~~ \int_{\mathbb{D}\times\mathbb{D}} \left|\int_{\mathbb{D}\times\mathbb{D}} \left[\frac{\partial^{\beta_1}}{\partial w_{1}^{\beta_1}}\left(\frac{1}{(1-w_{1}\bar{\eta}_1)^{2}}\right)\cdot \frac{\partial^{\beta_2}}{\partial w_{2}^{\beta_2}}\left(\frac{1}{(1-w_{2}\bar{\eta}_2)^{2}}\right)\right]f(\Phi(\eta))(\eta_1-\eta_2)\,dv(\eta)\right|^{p} \\
&~~~~ \cdot |w_1-w_2|^{2} dv(w)\\
&=\int_{\mathbb{D}\times\mathbb{D}}\left|\frac{1}{w_1^{\beta_1} w_2^{\beta_2}}\int_{\mathbb{D}\times\mathbb{D}}K_{\beta_1}(w_1,\eta_1) K_{\beta_2}(w_2,\eta_2)T_{\eta_1}^{\beta_1} T_{\eta_2}^{\beta_2}\left(f(\Phi(\eta))(\eta_1-\eta_2)\right)\,dv(\eta)\right|^{p} \\
&~~~~\cdot |w_1-w_2|^{2}dv(w)\\
& \lesssim \int_{\mathbb{D}\times\mathbb{D}}
\left|\int_{\mathbb{D}} \left| \frac{1}{(1-w_2\bar{\eta}_2)^{2}}\right|\int_{\mathbb{D}} \left| \frac{1}{(1-w_1\bar{\eta}_1)^{2}}\right|\left|T_{\eta_1}^{\beta_1} T_{\eta_2}^{\beta_2}\left(f(\Phi(\eta))(\eta_1-\eta_2)\right)\right|\,dA(\eta_1)\,dA(\eta_2)\right|^{p} \\
&~~~~ \cdot |w_1 - w_2|^2 dv(w)\\ 
&=\int_{\mathbb{D}\times\mathbb{D}}
\left| \mathcal{B}_{\mathbb{D}, \eta_1}^{+}  \mathcal{B}_{\mathbb{D}, \eta_2}^{+}
\left|T_{\eta_1}^{\beta_1} T_{\eta_2}^{\beta_2}\left(f(\Phi(\eta))(\eta_1-\eta_2)\right)\right|\,\right|^{p}  \cdot |w_1 - w_2|^2 dv(w)\\   
&  \lesssim \int_{\mathbb{D}} \int_{\mathbb{D}}
 \left| \mathcal{B}_{\mathbb{D}, \eta_2}^{+}
\left|T_{\eta_1}^{\beta_1} T_{\eta_2}^{\beta_2}\left(f(\Phi(\eta))(\eta_1-\eta_2)\right)\right|\,\right|^{p}  \cdot |w_1 - w_2|^2 dv(w_1) dv(w_2)\\ 
&  \lesssim \int_{\mathbb{D}}\int_{\mathbb{D}}
|T_{w_1}^{\beta_1} T_{w_2}^{\beta_2}\left(f(\Phi(w))(w_1-w_2)\right)|^{p} |w_1 - w_2|^2 dv(w_2)dv(w_1)\\
& \leq \sum_{|s|\leq\beta_1} \sum_{|q|\leq\beta_2} \int_{\mathbb{D}\times\mathbb{D}}
|D_{w_1,\bar{w}_1}^{s} D_{w_2,\bar{w}_2}^{q} (f\circ\Phi(w))|^{p}|w_1 - w_2|^2 dv(w)\\
& \lesssim \sum_{|\alpha|\leq k}\int_{\mathbb{G}} |D_{z,\bar{z}}^{\alpha} (f)|^{p} dv(z)=\|f\|^p_{W^{k,p}(\mathbb{G})},
\end{aligned}
\end{equation}
where the first equality follows from Corollary \ref{3cor1}, the first inequality follows from (\ref{4eq6}), the second and the third inequalities follow from the assumption on $\mathcal{B}^+_{\mathbb D}$, and the last inequality follows from Lemma \ref{2lem1}.

For the first term on the right hand side of (\ref{4eq4}), by Bell's transformation formula, 
\begin{equation}\label{Bell}
\mathcal{B}_{\mathbb{D}\times\mathbb{D}}(J_{\mathbb{C}} \Phi\cdot (h\circ\Phi ))=J_{\mathbb{C}} \Phi \cdot (\mathcal{B}_{\mathbb{G}}(h)\circ\Phi)
\end{equation} where $h\in L^2(\mathbb{G})$ (\cite{Bel}), one sees
\begin{equation}\label{addequ}
\begin{split}
&~~~~\int_{\mathbb{G}} |\mathcal{B}_\mathbb{G}(f)|^{p}|z_{1}^{2}-4z_2|^{\frac{3kp}{2}}\,d\,v(z) \\
&= \int_{\mathbb{D}\times \mathbb{D}}  \left|\mathcal{B}_{\mathbb{D}} \mathcal{B}_{\mathbb{D}} \left(  (w_1-w_2) (f\circ\Phi )(w)   \right)   \right|^p |w_1-w_2|^{(3k-1)p+2} dv(w)\\
&\lesssim  \int_{\mathbb{D}\times \mathbb{D}}  \left(\mathcal{B}^+_{\mathbb{D}} \mathcal{B}^+_{\mathbb{D}} \left(  \left| f\circ\Phi\right|   \right)  \right)^p |w_1-w_2|^{2} dv(w)\\
&\lesssim \int_{\mathbb{D}\times \mathbb{D}} \left| f\circ\Phi\right|  ^{p}  |w_1-w_2|^{2}  \,d\,v(w) \\
&= \int_{\mathbb{G}} |f|^{p}\,d\,v(z)=\|f\|_{L^{p}(\mathbb{G})}^{p}.
\end{split}
\end{equation}
The first inequality is due to the boundness of $|w_{1}-w_2|$ and $3k-1\geq 0$, and the second inequality follows  from the assumption on $\mathcal{B}^+_{\mathbb D}$ as in the proof of (\ref{111122}). 
Combining two parts, Theorem \ref{1thm0} is proved. 
\end{proof}

{\it Proof of Corollary \ref{Lp}}. The first part follows by combining Theorem \ref{1thm0} and Proposition \ref{5cor1}. The second part follows by combining Proposition \ref{1cor1} and  Proposition \ref{1thm00} below. For the third part, when $\frac{4}{3} < p < 4$, it follows from Part 2 since the weight function $e^{-\frac{p\delta}{2}}$ is bounded from above. When $p>2$, it follows from Part 1. The only difference is that it suffices to have $|z_{1}^{2}-4z_2|^{\frac{p}{2}}=|w_1-w_2|^{2p}$ as the weight function in (\ref{addequ}) to cancel out the term of $|J_{\mathbb{C}} \Phi |^{-p}=|w_1-w_2|^{-p}$ arisen in (\ref{Bell}). 
\qed

\begin{pro}\label{1thm00}
If the norm of $\mathcal{B}_{\mathbb{D}}: L^p \left(\mathbb{D}, |w_1-w_2|^{2-p}\right) \rightarrow L^p\left(\mathbb{D}, |w_1-w_2|^{2-p}\right)$ is uniformly bounded independent of $w_2\in \mathbb{D}$ for some $p \in (1, \infty)$, then the Bergman projection $\mathcal{B}_{\mathbb G}$ on the symmetrized bidisk $\mathbb{G}$ is bounded from $L^p(\mathbb{G})$ to itself.
\end{pro}
The proposition is implicit proved in \cite{CKY} (cf. Theorem 3.1) and
 we also include the proof here for the completeness. 
\begin{proof}
By Bell's transformation formula (\ref{Bell}), to prove the $L^p$ boundness of $\mathcal{B}_{\mathbb{G}}$:
$$\|\mathcal{B}_{\mathbb{G}}(h)\|_{L^p(\mathbb{G})}\lesssim \|h\|_{L^p(\mathbb{G})}$$ for $h\in L^p(\mathbb{G})$, it is equivalent to prove 
$$\int_{\mathbb{D}\times\mathbb{D}}\left|\mathcal{B}_{\mathbb{D}\times\mathbb{D}}\left(J_{\mathbb{C}} \Phi\cdot (h\circ\Phi )\right)\cdot \left(J_{\mathbb{C}} \Phi\right)^{-1}\right|^{p}\cdot \left| J_{\mathbb{C}} \Phi\right|^2 dV\lesssim \int_{\mathbb{D}\times\mathbb{D}}\left| h\circ\Phi \right|^p \left| J_{\mathbb{C}} \Phi\right|^2 dV.$$ Let $g=J_{\mathbb{C}} \Phi\cdot (h\circ\Phi )$, it suffices to prove that, for any $g\in L^p (\mathbb{D}\times\mathbb{D}, \left| J_{\mathbb{C}} \Phi\right|^{2-p})$,
$$\int_{\mathbb{D}\times\mathbb{D}}\left|\mathcal{B}_{\mathbb{D}\times\mathbb{D}}(g)\right|^p \cdot\left| J_{\mathbb{C}} \Phi\right|^{2-p} dV\lesssim \int_{\mathbb{D}\times\mathbb{D}}|g|^p \left| J_{\mathbb{C}} \Phi\right|^{2-p} dV.$$
By plugging in $J_{\mathbb{C}} \Phi$ and applying the Fubini's Theorem, one obtains
\begin{equation*}
\begin{aligned}
\int_{\mathbb{D}\times\mathbb{D}}\left|\mathcal{B}_{\mathbb{D}\times\mathbb{D}}(g)\right|^p \cdot|w_1-w_2|^{2-p} dV
& = \int_{\mathbb{D}}\int_{\mathbb{D}}\left|\mathcal{B}_{\mathbb{D}}\mathcal{B}_{\mathbb{D}}(g(w_1,w_2))\right|^p\cdot |w_1-w_2|^{2-p} \,dA(w_1)dA(w_2)\\
& \lesssim  \int_{\mathbb{D}}\int_{\mathbb{D}}|g(w_1,w_2)|^p\cdot |w_1-w_2|^{2-p} \,dA(w_1)dA(w_2)
\end{aligned}
\end{equation*}
The last inequality is due to the boundedness of $\mathcal{B}_{\mathbb{D}}$ from $L^p \left(\mathbb{D}, |w_1-w_2|^{2-p}\right)$ to itself, the independence and the symmetry in $w_1, w_2$. 
\end{proof}

{\it Acknowlegement:} We thank Loredana Lanzani for encouraging us to clarify the $L^p$ boundedness of the Bergman projection on the symmetrized polydisk.

 \noindent Liwei Chen, liwei.chen@ucr.edu, Department of Mathematics, University of California, Riverside, CA 92521 USA.\\
 \noindent Muzhi Jin, mujin@syr.edu, Department of Mathematics, Syracuse University, Syracuse, NY 13244 USA.\\
 \noindent Yuan Yuan, yyuan05@syr.edu, Department of Mathematics, Syracuse University, Syracuse, NY 13244 USA.\\

\end{document}